\pdfoutput=1
\documentclass[12pt]{article}
\usepackage[utf8]{inputenc}
\usepackage[british]{babel}
\usepackage{cmap}
\usepackage{lmodern}

\usepackage{amssymb, amsmath, amsthm}
\usepackage[a4paper,top=25mm,bottom=25mm,left=25mm,right=25mm]{geometry}
\usepackage{ragged2e}

\usepackage{authblk} % for headings
\usepackage{pifont}
\usepackage{graphicx}
\usepackage[dvipsnames,svgnames,table]{xcolor}
\usepackage[figuresright]{rotating}
\usepackage{xtab} % tackle the long tables
\usepackage{longtable} % tackle the long tables
\usepackage{multirow}
\usepackage{footnote}
\usepackage[stable]{footmisc}
\usepackage{chngpage} % allows for temporary adjustment of side margins
\usepackage{pdflscape} % landscape environment
\usepackage[nottoc,notlot,notlof]{tocbibind} % includes everything in ToC

\usepackage{pgfplots}
\pgfplotsset{every tick label/.append style={font=\footnotesize}}
\pgfplotsset{compat=1.18}
\usepackage{setspace}

\usepackage{array}
\newcolumntype{K}[1]{>{\centering\arraybackslash$}p{#1}<{$}}

\makesavenoteenv{tabular}
\usepackage{tabularx}
\usepackage{booktabs}
\usepackage{threeparttable}
\usepackage[referable]{threeparttablex} % footnotes in tabu
\newcolumntype{R}{>{\raggedleft\arraybackslash}X}
\newcolumntype{L}{>{\raggedright\arraybackslash}X}
\newcolumntype{C}{>{\centering\arraybackslash}X}
\newcolumntype{A}{>{\columncolor{gray!25}}C}
\newcolumntype{a}{>{\columncolor{gray!25}}c}

\newlength{\tablen}

\usepackage{dcolumn} % alignment to decimal points
\newcolumntype{.}{D{.}{.}{-1}}

\usepackage{tikz}
\usetikzlibrary{arrows, calc, matrix, patterns, positioning, shapes, trees}
\usepackage[semicolon]{natbib} % numbers
\usepackage[hyphens]{url}
\usepackage{hyperref} % [hidelinks]
\hypersetup{
  colorlinks   = true,    		% Colours links instead of ugly boxes
  urlcolor     = blue,    		% Colour for external hyperlinks
  linkcolor    = blue,    		% Colour of internal links
  citecolor    = ForestGreen	% Colour of citations
}
\usepackage{microtype}
\usepackage[justification=centering]{caption} % [justification=centering]

\usepackage[labelformat=simple]{subcaption}

\DeclareCaptionLabelFormat{parenthesis}{(#2)}
\makeatletter
\renewcommand\p@subfigure{\arabic{figure}.}
\makeatother

\DeclareCaptionLabelFormat{parenthesis}{(#2)}
\makeatletter
\renewcommand\p@subtable{\arabic{table}.}
\makeatother

\usepackage{enumitem}

\setlist[itemize]{leftmargin=2.5\parindent}
\setlist[enumerate]{leftmargin=2.5\parindent}

\def\addlegendimage{\csname pgfplots@addlegendimage\endcsname}

\theoremstyle{plain}
\newtheorem{corollary}{Corollary}%[section]

\newtheorem{proposition}{Proposition}%[section]
\newtheorem{theorem}{Theorem}%[section]

\theoremstyle{definition}
\newtheorem{example}{Example}%[section]

\theoremstyle{remark}

\newtheorem{remark}{Remark}%[section]

\makeatletter
\let\@fnsymbol\@alph
\makeatother

\def\keywords{\vspace{.5em} % Add keywords
{\noindent \textit{Keywords}: }}

\def\AMS{\vspace{.5em} % Add keywords
{\noindent \textbf{\emph{MSC} class}: }}

\def\JEL{\vspace{.5em} % Add keywords
{\noindent \textbf{\emph{JEL} classification number}: }}

\title{The logarithmic least squares priorities and \\ ordinal violations in the best-worst method}
\author{\href{https://sites.google.com/view/laszlocsato}{L\'aszl\'o Csat\'o}\thanks{~E-mail: \emph{laszlo.csato@sztaki.hun-ren.hu}} }
\affil{Institute for Computer Science and Control (SZTAKI) \\
Hungarian Research Network (HUN-REN) \\
Laboratory on Engineering and Management Intelligence \\
Research Group of Operations Research and Decision Systems}
\affil{Corvinus University of Budapest (BCE) \\
Institute of Operations and Decision Sciences \\
Department of Operations Research and Actuarial Sciences}
\affil{Budapest, Hungary}
\date{\today}

\def\Dedication{
\begin{small}
{\noindent
``\emph{Statistical results show that BWM performs significantly better than AHP with respect to the consistency ratio, and the other evaluation criteria: minimum violation, total deviation, and conformity.}''\footnote{~Source: \citet[p.~49]{Rezaei2015}.}
}
\end{small}

\vspace{0.5cm} 
\justify }

\begin{document}

\maketitle
\thispagestyle{empty}
\Dedication

\begin{abstract}
\noindent
The best-worst method is an increasingly popular approach to solving multi-criteria decision-making problems. However, the usual prioritisation techniques may result in an ordinal violation if the best (worst) alternative identified in the first step does not receive the highest (lowest) weight.
The current paper gives two sufficient conditions for the logarithmic least squares method, applied to an incomplete best-worst method matrix, to guarantee the lack of ordinal violations. Our results provide another powerful argument for using the logarithmic least squares priorities in the best-worst method.
% ORL 80 words
%The best-worst method is an increasingly popular approach to solve multi-criteria decision-making problems. However, the usual prioritisation techniques may result in ordinal violations if the best (worst) alternative identified in an early step does not receive the highest (lowest) weight in a later step.
%The current paper gives two sufficient conditions for the logarithmic least squares method, applied to the implied incomplete best-worst method matrix, to guarantee the lack of ordinal violations.


\AMS{15A06, 90B50, 91B08}
% Linear equations (linear algebraic aspects)
% Management decision making, including multiple objectives
% Individual preferences

\JEL{C44, D71}
% Operations Research, Statistical Decision Theory
% Social Choice, Clubs, Committees, Associations
\end{abstract}

\clearpage

\section{Introduction} \label{Sec1}

Pairwise comparisons are widely used in multi-criteria decision-making (MCDM) if the decision-maker cannot directly determine the weights of the alternatives/criteria. For example, the Analytic Hierarchy Process (AHP), one of the most popular MCDM methodologies, is based on pairwise comparisons \citep{Saaty1980}. While focusing on pairwise comparisons allows the decomposition of the original problem into subproblems that are easier to deal with, this simplification has a price. First, the pairwise comparisons may be inconsistent, making the derivation of the priorities far from trivial. Second, the number of pairwise comparisons needed for $n$ alternatives is $n(n-1)/2$, which can be costly and cumbersome to obtain if $n$ is relatively high.

To address these problems, \citet{Rezaei2015} has introduced the best-worst method: after the identification of the best (most desirable, most important) and the worst (least desirable, least important) alternatives, all other alternatives are compared to the best and the worst alternatives, respectively. This approach requires only $2n-3$ comparisons. Due to its efficiency in reducing the number of pairwise comparisons and its ability to maintain consistency, the best-worst method has attracted the attention of several researchers and has been applied to solve many real-world problems \citep{MiTangLiaoShenLev2019}.

\subsection{Prioritisation techniques in the best-worst method} \label{Sec11}

% However, the usual prioritisation techniques in the best-worst method may suffer from various shortcomings. Some procedures fail to take indirect comparisons into account \citep{BrunelliRezaei2019, Rezaei2015, Rezaei2016}, have multiple optimal solutions \citep{BrunelliRezaei2019, Rezaei2015}, and can contain ordinal violations \citep{BrunelliRezaei2019, Rezaei2015, Rezaei2016, TuWuPedrycz2023}.

Besides introducing the best-worst method, \citet{Rezaei2015} suggested the least worst absolute error method to derive the priority vector. However, it has two shortcomings. First, it is a nonconvex optimisation problem without a closed-form solution. Second, it might have multiple optimal solutions. To address the first issue, \citet{BrunelliRezaei2019} defined a new metric, which is mathematically more sound and leads to an optimisation problem that can be simply linearised and solved. However, similar to the method of \citet{Rezaei2015}, the approach of \citet{BrunelliRezaei2019} usually has multiple optimal solutions \citep{TuWuPedrycz2023}. Hence, the most popular prioritisation technique for the best-worst method is the linear model of \citet{Rezaei2016}, which guarantees the uniqueness of the weights.

Nonetheless, the above procedures \citep{BrunelliRezaei2019, Rezaei2015, Rezaei2016} fail to take indirect comparisons into account.
Thus, recent studies have proposed to consider the implied incomplete pairwise comparison matrix, and derive priorities from this best-worst method matrix by the well-established techniques of multi-criteria decision-making \citep{TuWuPedrycz2023, XuWang2024}.
In particular, the logarithmic least squares method is found to be a reasonable technique since it
(a) is simple to calculate as the solution of a linear system of equations \citep{BozokiFulopRonyai2010};
(b) leads to unique weights \citep{BozokiFulopRonyai2010}; and
(c) accounts for indirect comparisons, too \citep{TuWuPedrycz2023}.
Furthermore, \citet{XuWang2024} show that the corresponding priority vector is efficient, which is a crucial property of the weights \citep{BlanqueroCarrizosaConde2006, BozokiFulop2018, FurtadoJohnson2024b, SzadoczkiBozoki2024}.

\subsection{The issue of ordinal violations} \label{Sec12}

The best-worst method requires pairwise comparisons concerning the best and the worst criteria. However, the priorities derived from a best-worst method matrix might exhibit an ordinal violation \citep{TuWuPedrycz2023}: the best (worst) alternative identified in the first stage does not receive the highest (lowest) weight in the final stage, which necessitates the re-examination of the pairwise comparisons.

According to our knowledge, the issue of ordinal violation has been considered first for the best-worst method by \citet{TuWuPedrycz2023}. Nevertheless, it has been widely analysed in the literature on pairwise comparison matrices since the pioneering work of \citet{GolanyKress1993}.
\citet{SirajMikhailovKeane2012b} present a heuristic algorithm that improves ordinal consistency by identifying and eliminating intransitivities in pairwise comparison matrices.
\citet{ChenKouLi2018} propose a procedure based on linear programming to update the weight vector obtained by popular prioritisation methods with a better one to minimise the number of rank violations.
\citet{FaramondiOlivaBozoki2020} extend the logarithmic least squares method for incomplete pairwise comparison matrices \citep{BozokiFulopRonyai2010} with a procedure composed of two complementary steps. The first stage maximises the number of ordinal constraints satisfied, while the second phase optimises cardinal preferences with additional restrictions determined in the first stage.
\citet{TuWu2021} study the elimination of rank violations in fuzzy preference relations by deriving a system of constraints to ensure that this requirement is explicitly controlled in the optimisation model.
\citet{WangPengKou2021} design a two-stage ranking method in order to minimise ordinal violations, the number of conflicts between the ranking and the dominance relationship for pairwise comparison matrices.
\citet{YuanWuTu2023} suggest an optimisation model that considers both cardinal consistency and ordinal consistency to estimate unknown preferences in the case of incomplete fuzzy preference relations.
\citet{Csato2024b} shows that the lexicographically optimal completion of missing pairwise comparisons \citep{AgostonCsato2024} combined with any reasonable weighting method eliminates all ordinal violations for incomplete pairwise comparison matrices represented by a weakly connected directed acyclic graph.

\subsection{The research gap} \label{Sec13}

It is already known that the usual prioritisation techniques of the best-worst method \citep{BrunelliRezaei2019, Rezaei2015, Rezaei2016} can lead to an ordinal violation \citep[Table~4]{TuWuPedrycz2023}. In particular, the Monte Carlo simulation method of \citet{TuWuPedrycz2023} shows the following probabilities of an ordinal violation:
(1) 0.95\% for the preference weighted least worst absolute error of \citet{Rezaei2016};
(2) 14.97\% for the least worst absolute error of \citet{Rezaei2015}; and
(3) 19.24\% for the multiplicative optimisation problem of \citet{BrunelliRezaei2019} that can be easily linearised.

However, the thorough numerical experiment of \citet{TuWuPedrycz2023} fails to identify any best-worst method matrix where the logarithmic least squares method shows an ordinal violation. This implies important questions that have not been answered by the existing literature: What conditions can guarantee the lack of ordinal violation by the logarithmic least squares method? What is the probability that an ordinal violation emerges in a best-worst method matrix?

As an analogy, consider Pareto (in)efficiency of Saaty's eigenvector method \citep{Saaty1980} for multiplicative pairwise comparison matrices. \citet{BlanqueroCarrizosaConde2006} have revealed that the eigenvector method does not always yield an efficient weight vector. According to \citet{Bozoki2014a}, inefficiency occurs for a class of pairwise comparison matrices with arbitrarily small inconsistency. In recent years, several sufficient conditions have been proved for the efficiency of the right Perron eigenvector \citep{Abele-NagyBozoki2016, Abele-NagyBozokiRebak2018, daCruzFernadesFurtado2021, FernandesFurtado2022, FernandesPalheira2024, Furtado2023, FurtadoJohnson2024c}, and further families of matrices for which the Perron vector is inefficient have also been found \citep{FurtadoJohnson2024d}.

\subsection{Main contributions} \label{Sec14}

The current paper investigates ordinal violations in the best-worst method.
We provide sufficient conditions for the logarithmic least squares method to guarantee the lack of ordinal violations in a best-worst method matrix. The relatively general conditions contain a uniform lower bound for the dominance of the best alternative over all other alternatives, as well as for the dominance of all other alternatives over the worst alternative. The maximal numerical preference is restricted by a polynomial of this uniform lower bound.

Our first theorem explains why the logarithmic least squares method shows no ordinal violation in the dataset analysed by \citet{TuWuPedrycz2023}: their procedure used to generate random best-worst method matrices almost satisfies the conditions derived here. Nonetheless, the proof of the theorem uncovers some best-worst method matrices for which the logarithmic least squares method does not avoid an ordinal violation.

According to our second result, if the Saaty scale is applied to a best-worst method matrix and the best alternative is preferred to the worst alternative at least as much as the preference between any two alternatives, then the logarithmic least squares priorities are always compatible with the preference order given by the decision-maker if the number of alternatives does not exceed 26.
%However, this is not true for the other four prioritisation techniques suggested by \citet{BrunelliRezaei2019, Rezaei2015, Rezaei2016, TuWuPedrycz2023}.

\begin{figure}[t!]
\centering

\begin{tikzpicture}[scale=1,auto=center, transform shape, >=triangle 45]
\tikzstyle{every node}=[draw,align=center];
  \node (N1) at (0,11) {Identifying the best and \\ the worst alternatives};
  \node (N2) at (0,9) {Comparing all alternatives \\ to the best and the worst};
  \node (N3) at (0,7) {Constructing the (incomplete) \\ best-worst method matrix};
  \node (N4) at (0,5) {Calculating the priorities};
  \node[shape = ellipse] (N5) at (0,3) {Is there any \\ ordinal violation?};
  \node (N6) at (0,0) {The result is \\ acceptable};
  \node (N7) at (5,0) {The result is \\ contradictory};

\tikzstyle{every node}=[align=center];
  %\node (N0) at (5,11) {};  
  \draw [->,line width=1pt] (N1) -- (N2)  {};
  \draw [->,line width=1pt] (N2) -- (N3)  {};
  \draw [->,line width=1pt] (N3) -- (N4)  {};
  \draw [->,line width=1pt] (N4) -- (N5)  {};
  \draw [->,line width=1pt] (N5) -- (N6)  node [midway, left] {No};
  \draw [->,line width=1pt] (N5) -- (N7)  node [near start, above right] {Yes};
  \draw [line width=1pt] (N7) -- (5,11);
  \draw [->,line width=1pt] (5,11) -- (N1) {};
  \draw [line width=1pt,dashed] (3,12) -- (7,12) -- (7,-1) -- (3,-1) -- (3, 12) {};
  \node at (9.5,5.5) {\textbf{Our results give} \\ \textbf{sufficient conditions} \\ \textbf{to avoid this loop} \\ \textbf{by guaranteeing} \\ \textbf{the lack of any} \\ \textbf{ordinal violation}};
\end{tikzpicture}
\captionsetup{justification=centering}
\caption{Our main contribution to the theoretical background of the best-worst method}
\label{Fig1}
\end{figure}

Figure~\ref{Fig1} presents a graphical representation of the significance of these results. The left-hand side outlines the main steps of the best-worst method. However, if the derived priorities contain an ordinal violation, then the feedback mechanism indicated on the right-hand side should be used, and the necessary reconsideration of the preferences involves further interaction with the decision-maker who thinks the task is already finished. The conditions provided in Theorems~\ref{Theo1} and \ref{Theo2} guarantee that this additional loop is never reached.

\subsection{The significance of the results and their practical relevance} \label{Sec15}

The best-worst method has been developed to substantially reduce the number of pairwise comparisons required to solve a multi-criteria decision-making problem. However, the derived priorities may contain an ordinal violation that could be disturbing and difficult to accept by the decision-maker: sometimes, either the originally identified best alternative does not have the highest weight, or the originally identified worst alternative does not have the lowest weight. Such a result calls for a cumbersome re-examination of the preferences, which seriously threatens the main advantage of this approach, its simplicity.

Our theorems provide sufficient conditions to eliminate the possibility of an ordinal violation if the logarithmic least squares method is applied to a best-worst method matrix. These constraints can be directly built into any decision-making software to guarantee the lack of ordinal violations. Hence, imposing the additional restrictions derived in the paper will prevent the use of an extensive feedback mechanism that is needed to treat a severe shortcoming of the most popular prioritisation techniques suggested for the best-worst method.

\subsection{Structure} \label{Sec16}

The paper is organised as follows.
Section~\ref{Sec2} presents the logarithmic least squares priority deriving technique in the context of the best-worst method. The main results on ordinal violations are verified in Section~\ref{Sec3}. Finally, Section~\ref{Sec4} discusses our findings and concludes.

\section{The logarithmic least squares priorities in a best-worst method matrix} \label{Sec2}

A \emph{pairwise comparison matrix} $\mathbf{A} = \left[ a_{ij} \right]$ is an $n \times n$ positive ($a_{ij} > 0$ for all $1 \leq i,j \leq n$) matrix that satisfies reciprocity: $a_{ij} a_{ji} = 1$ for all $1 \leq i,j \leq n$. As the number of pairwise comparisons is a quadratic function of the number of alternatives $n$, there are several suggestions to reduce the number of questions asked by the decision-maker. In the best-worst method (BWM), only the pairwise comparisons concerning the best ($B$) and the worst ($W$) alternatives should be obtained \citep{Rezaei2015}. This results in a \emph{best-worst method matrix} $\mathbf{A} = \left[ a_{ij} \right]$, where $a_{Bj}$, $a_{jB}$, $a_{Wj}$, and $a_{jW}$ are known for all $1 \leq j \leq n$.

Following \citet[Section~5.2]{TuWuPedrycz2023}, we assume in the following that $a_{Bj} > 1$ and $a_{jW} < 1$ for all $j \neq B,W$. Furthermore, the matrix entries are real numbers---but not necessarily integers or reciprocals of integers---between 1/9 and 9 in accordance with the recommendation of Saaty \citep{Saaty1980}.

Pairwise comparisons are primarily used to derive a reliable priority vector. This is trivial if the matrix is consistent, that is, $a_{ij} a_{jk} = a_{ik}$ for all $1 \leq i,j,k \leq n$, when a unique priority vector $\mathbf{w} = \left[ w_i \right]$ exists such that $a_{ij} = w_i / w_j$ holds for all $1 \leq i,j \leq n$.
However, pairwise comparison matrices are usually \emph{inconsistent}. \citet{Brunelli2018} and \citet{KulakowskiTalaga2020} survey inconsistency measures used for complete and incomplete pairwise comparison matrices, respectively.

The most prominent inconsistency index is the inconsistency ratio $\mathit{CR}$ proposed by Saaty \citep{Saaty1980}. It is a linear transformation of the dominant eigenvalue of the pairwise comparison matrix. The value of $\mathit{CR}$ depends on the so-called random index $\mathit{RI}$ that has been calculated for complete pairwise comparison matrices \citep{BozokiRapcsak2008}, as well as for incomplete matrices \citep{AgostonCsato2022} and, specifically, for the best-worst method \citep{TuWuPedrycz2023}. 

Analogously, a number of prioritisation techniques have been suggested in the literature for complete pairwise comparison matrices \citep{ChooWedley2004}. One of the most popular procedures is the \emph{logarithmic least squares method} (LLSM) \citep{CrawfordWilliams1985, WilliamsCrawford1980}, partially due to its favourable axiomatic properties \citep{Barzilai1997, BozokiTsyganok2019, Csato2018b, Csato2019a, Fichtner1984, Fichtner1986, LundySirajGreco2017}.
LLSM has been extended to the case of incomplete pairwise comparison matrices \citep{Kwiesielewicz1996, TakedaYu1995}.

According to \citet[Theorem~4]{BozokiFulopRonyai2010}, the optimal solution of the incomplete LLSM problem is unique if the undirected graph associated with the pairwise comparison matrix is connected (which is guaranteed in the best-worst method) and can be calculated as follows:
\begin{equation} \label{eq_LLSM}
\mathbf{L} \mathbf{y} = \mathbf{r} \qquad \text{and} \qquad \sum_{i=1}^n y_i = 0,
\end{equation}
where
\begin{itemize}
\item
$\mathbf{L} = \left[ \ell_{ij} \right]$ is the Laplacian matrix of the associated graph ($\ell_{ii}$ is the degree of node $i$; $\ell_{ij} = -1$ if $i \neq j$ and $a_{ij}$ is known; $\ell_{ij} = 0$ if $i \neq j$ and $a_{ij}$ is unknown);
\item
$\mathbf{y} = \left[ y_i = \log w_i \right]$ is the vector derived by taking the elementwise logarithm of the priority vector $\mathbf{w} = \left[ w_i \right]$; and
\item
$\mathbf{r} = \left[ r_i = \log \left( \prod_{j=1}^n a_{ij} \right) \right]$ is the vector given by logarithms of the multiplied entries in each row.
\end{itemize}
Note that \citet[Theorem~4]{BozokiFulopRonyai2010} uses the normalisation $w_n = 1$ ($y_i = \log w_i = 0$) instead of $\sum_{i=1}^n y_i = 0$ to treat the singularity of Laplacian matrix $\mathbf{L}$, and delete the last columns and row of $\mathbf{L}$. We think that the equivalent normalisation of $\sum_{i=1}^n y_i = 0$ \citep[Lemma~2]{Csato2015a} is more convenient as it does not differentiate between the alternatives.

Formula~\eqref{eq_LLSM} has appeared in \citet{KaiserSerlin1978} and \citet{CaklovicKurdija2017}, too, independently of \citet{BozokiFulopRonyai2010}.

The logarithmic least squares method has recently been suggested as a promising technique to derive priorities from a best-worst method matrix in \citet{TuWuPedrycz2023} and \citet{XuWang2024}.

A priority vector $\mathbf{w} = \left[ w_i \right]$ shows an \emph{ordinal violation} of preferences if $w_i \leq w_j$ but the pairwise comparison matrix contains the entry $a_{ij} > 1$. The number of violations has been introduced by \citet{GolanyKress1993} as an important criterion to evaluate prioritisation techniques, and has been widely accepted in the literature \citep{ChenKouLi2018, Csato2024b, CsatoRonyai2016, FaramondiOlivaBozoki2020, TuWu2021, TuWu2023, TuWuPedrycz2023, YuanWuTu2023, WangPengKou2021}.

\section{The main results} \label{Sec3}

In the following, two theorems are provided that guarantee the lack of any ordinal violation if the logarithmic least squares method is applied to a best-worst method (incomplete pairwise comparison) matrix. Both contain a lower bound for the preferences between the best alternative and all other alternatives, as well as between all other alternatives and the worst alternative. In other words, the best (worst) alternative should be ``sufficiently'' good (bad). Furthermore, a higher bound is required for all pairwise comparisons to avoid that an alternative is excessively dominant over (dominated by) the worst (best) alternative.

The proofs strongly exploit the closed-form solution of the logarithmic least squares method, given by a system of linear equations. Both the mathematical equivalence of the geometric mean of weight vectors calculated from all spanning trees and the logarithmic least squares problem \citep{BozokiTsyganok2019} and the determination of uncertainty bounds for pairwise comparisons \citep{FaramondiOlivaSetolaBozoki2023a} are based on this favourable property of the incomplete logarithmic least squares method.

\begin{theorem} \label{Theo1}
Assume that the best alternative is at least $p$ times better than all other alternatives, and the worst alternative is at least $p$ times worse than all other alternatives. Furthermore, the maximal numerical preference is at most $p^3$.
Then the logarithmic least squares priorities do not contain any ordinal violation in a best-worst method matrix.
\end{theorem}

\begin{proof}
Note that the best and worst alternatives are compared to all other alternatives, thus, $\ell_{BB} = \ell_{WW} = n-1$ in equation \eqref{eq_LLSM}. On the other hand, $\ell_{jj} = 2$ for any third alternative $j \neq B,W$.
 
Due to \eqref{eq_LLSM}, the logarithmic least squares priorities satisfy the following equation for the best alternative $B$:
\begin{equation} \label{eq_Proof1}
(n-1) y_B - y_W - \sum_{k \neq B,W} y_k = \log \left( \prod_{k=1}^n a_{Bk} \right)
\end{equation}
Due to \eqref{eq_LLSM}, the logarithmic least squares priorities satisfy the following equation for the worst alternative $W$:
\begin{equation} \label{eq_Proof2}
(n-1) y_W - y_B - \sum_{k \neq B,W} y_k = \log \left( \prod_{k=1}^n a_{Wk} \right)
\end{equation}
Due to \eqref{eq_LLSM}, the logarithmic least squares priorities satisfy the following equation for any third alternative $j \neq B,W$:
\begin{equation} \label{eq_Proof3}
2y_j - y_B - y_W = \log \left( a_{jB} a_{jW} \right).
\end{equation}
Since the sum of the logarithmic weights is normalised to 0 in \eqref{eq_LLSM}, by adding $\sum_{i=1}n y_i = 0$ to the expression on the right hand-side of \eqref{eq_Proof1} and \eqref{eq_Proof2}, respectively:
\begin{equation} \label{eq_Proof4}
n y_B = \log \left( \prod_{k=1}^n a_{Bk} \right);
\end{equation}
\begin{equation} \label{eq_Proof5}
n y_W = \log \left( \prod_{k=1}^n a_{Wk} \right).
\end{equation}

An ordinal violation occurs if $y_j > y_B$ or $y_j < y_W$ for $j \neq B,W$.
Let us consider the first case when $y_j > y_B$. On the basis of equations~\eqref{eq_Proof3}, \eqref{eq_Proof4}, \eqref{eq_Proof5}, this is equivalent to
\begin{equation} \label{eq_Proof6}
\log \left( \prod_{k=1}^n a_{Bk} \right) < \log \left( \prod_{k=1}^n a_{Wk} \right) + n \log \left( a_{jB} a_{jW} \right).
\end{equation}
Since $a_{Wj} a_{jW} = 1$, we get
\begin{equation} \label{eq_Proof7}
\prod_{k=1}^n a_{Bk} < \left( \prod_{k=1, k \neq j}^n a_{Wk} \right) a_{jB}^n a_{jW}^{n-1}.
\end{equation}
If the best alternative is at least $p$ times better than all other alternatives, then the left hand side of~\eqref{eq_Proof7} is at least $p^{n-1}$.
If the worst alternative is at least $p$ times worse than all other alternatives and the maximal preference is smaller than $p^3$, then the right hand side of~\eqref{eq_Proof7} cannot be higher than
\begin{equation} \label{eq_Proof8}
\left( \frac{1}{p} \right)^{n-2} \left( \frac{1}{p} \right)^{n} p^{3(n-1)} = p^{n-1},
\end{equation}
which results in a contradiction.

The calculation for the second case of $y_j < y_W$ is analogous.
\end{proof}

\begin{remark} \label{Rem1}
Interestingly, the upper bound $p^3$ in Theorem~\ref{Theo1} does not depend on the number of alternatives $n$.
\end{remark}

The proof of Theorem~\ref{Theo1} helps construct a best-worst method matrix for which the logarithmic least squares method shows an ordinal violation.

\begin{example} \label{Examp1}
Assume that the best-worst method is used with the Saaty scale of
\[
\{ 1/9, 1/8, \dots ,1/2,1,2, \dots ,8,9 \}.
\]
In the matrix below, the first alternative is the best (which is better than all other alternatives by a factor of at least 2) and the last, sixth alternative is the worst (which is worse than all other alternatives by a factor of at least 2):
\[
\mathbf{A} = \left[
\begin{array}{K{2.5em} K{2.5em} K{2.5em} K{2.5em} K{2.5em} K{2.5em}}
    1     	 & 2	 	& 2    & 2 	  & 2    & 2 \\
    1/2		 & 1     	& \ast & \ast & \ast & 9 \\
    1/2      & \ast    	& 1    & \ast & \ast & 2 \\
    1/2      & \ast    	& \ast & 1    & \ast & 2 \\
    1/2      & \ast  	& \ast & \ast & 1    & 2 \\
    1/2      & 1/9    	& 1/2  & 1/2  & 1/2  & 1 \\
\end{array}
\right],
\]
where $\ast$ denotes a missing comparison as usual.
The priorities according to the logarithmic least squares method are:
\[
\mathbf{w} = \left[
\begin{array}{K{3em} K{3em} K{3em} K{3em} K{3em} K{3em}}
    0.2645 & 0.2778	& 0.1310 & 0.1310 & 0.1310 & 0.0648 \\
\end{array}
\right]^\top.
\]
Consequently, the second alternative has the highest weight even though the first has been identified as the best. Therefore, the preferences need to be reconsidered.
\end{example}

In the best-worst method, it is reasonable to assume that $a_{BW} \geq a_{jW}$ and $a_{BW} \geq a_{Bj}$ hold for any third alternative $j \neq B,W$, namely, the best alternative is preferred to the worst alternative at least as much as the preference between any two alternatives.

\begin{theorem} \label{Theo2}
Assume that the best alternative is at least $p$ times better than all other alternatives, and the worst alternative is at least $p$ times worse than all other alternatives. In addition, the best alternative is preferred to the worst alternative at least as much as the preference between any two alternatives. Finally, the maximal numerical preference is at most $p^{4 / (n-3) + 3}$.
Then the logarithmic least squares priorities do not contain any ordinal violation in a best-worst method matrix.
\end{theorem}

\begin{proof}
The steps of the proof of Theorem~\ref{Theo1} can be followed until formula \eqref{eq_Proof7} is reached:
\begin{equation} \label{eq_Proof9}
a_{BW} \prod_{k=1,k \neq W}^n a_{Bk} < a_{WB} \left( \prod_{k=1, k \neq j,W}^n a_{Wk} \right) a_{jB}^n a_{jW}^{n-1}.
\end{equation}
If the best alternative is at least $p$ times better than all other alternatives, then the left hand side of~\eqref{eq_Proof9} is at least $a_{BW} p^{n-2}$.
If the worst alternative is at least $p$ times worse than all other alternatives, then the right hand side of~\eqref{eq_Proof9} is at most
\begin{equation} \label{eq_Proof10}
a_{WB} \left( \frac{1}{p} \right)^{n-3} \left( \frac{1}{p} \right)^{n} a_{jW}^{n-1}.
\end{equation}
From \eqref{eq_Proof9}, we get
\begin{equation} \label{eq_Proof11}
a_{BW}^2 p^{3n-5} < a_{jW}^{n-1}.
\end{equation}
Since $a_{BW} \geq a_{jW}$, \eqref{eq_Proof11} implies
\begin{equation} \label{eq_Proof12}
p^{3n-5} < a_{jW}^{n-3}.
\end{equation}
However, $a_{jW} < p^{4/(n-3) + 3}$ according to the conditions of Theorem~\ref{Theo2}, which results in a contradiction.

The calculation for the second case of $y_j < y_W$ is analogous.
\end{proof}

\begin{remark} \label{Rem2}
Since $p^{4 / (n-3) + 3} \to p^3$ if $n \to \infty$, the additional restriction in Theorem~\ref{Theo2} compared to Theorem~\ref{Theo1} does not make ordinal violations less likely for best-worst method matrices with a high number of alternatives.
\end{remark}

\begin{remark} \label{Rem3}
Example~\ref{Examp1} does not satisfy the conditions of Theorem~\ref{Theo2} because the preference between the best and the worst alternatives (2) is weaker than the preference between the best (first) and the second alternatives (9).
\end{remark}

The following result uncovers that Theorem~\ref{Theo2} can be useful in practice, although it is not more powerful than Theorem~\ref{Theo1} if $n \to \infty$ according to Remark~\ref{Rem2}.

\begin{corollary} \label{Cor2}
Assume that the elements of a pairwise comparison matrix belong to the Saaty scale of $\{ 1/9, 1/8, \dots ,1/2,1,2, \dots ,8,9 \}$, and the best alternative is better than all other alternatives and the worst alternative is worse than all other alternatives. Furthermore, the best alternative is preferred to the worst alternative at least as much as the preference between any two alternatives. Then $p=2$ in Theorem~\ref{Theo2}, and our result guarantees the avoidance of ordinal violations if $2^{4/(n-3) + 3} > 9$, that is, for all $n \leq 26$.
\end{corollary}

\section{Discussion} \label{Sec4}

The current paper has proved two sufficient conditions for the logarithmic least squares priorities to avoid any violation of ordinal preferences in a best-worst method matrix. Our results provide a further argument for using the LLSM to derive priorities from best-worst pairwise comparison matrices since no similar guarantees exist for other prioritisation techniques. Thus, they may lead to ordinal violations, which require further interaction with the decision-maker as presented in the Introduction.

\subsection{A comparison of the main theorems} \label{Sec41}

In the previous literature, there are no analogous sufficient conditions that guarantee the lack of an ordinal violation. On the other hand, Theorems~\ref{Theo1} and \ref{Theo2} are worth comparing to each other. Both require a uniform lower bound $p$ on the preference between the best alternative and all other alternatives, as well as between all other alternatives and the worst alternative. The maximal numerical preference is allowed to be $p^3$ in Theorem~\ref{Theo1} but $p^{4/(n-3) + 3} > p^3$ in Theorem~\ref{Theo2}.
However, Theorem~\ref{Theo2} contains an additional constraint: the best alternative should be preferred to the worst alternative at least as much as the preference between any two alternatives, that is, the pairwise comparison between the best and worst alternatives needs to be the highest one. This supplementary restriction can be seen as the ``price'' to pay for having a less restrictive upper bound in Theorem~\ref{Theo2} than in Theorem~\ref{Theo1}.

\subsection{Finding ordinal violations via numerical experiments} \label{Sec42}

Recall that the thorough Monte Carlo experiment of \citet{TuWuPedrycz2023} has not resulted in any ordinal violation by the logarithmic least squares method. They have generated 10 thousand random best-worst pairwise comparison matrices with six alternatives based on the Saaty scale. It is possible to exactly determine the probability that an ordinal violation occurs under these assumptions.

\begin{proposition} \label{Prop1}
Assume that the elements of a best-worst method matrix of size six belong to the Saaty scale of $\{ 1/9, 1/8, \dots ,1/2,1,2, \dots ,8,9 \}$, the best alternative is better than all other alternatives and the worst alternative is worse than all other alternatives.
The number of different matrices with this property is $8^9 = 134{,}217{,}728$.
Among them, $7^9 = 40{,}353{,}607$ matrices (30.1\%) satisfy the sufficient conditions of Theorem~\ref{Theo1}.
The logarithmic least squares method yields an ordinal violation for only 56 matrices.
\end{proposition}

\begin{proof}
It can be assumed without loss of generality that the best alternative is the first and the worst is the sixth as in Example~\ref{Examp1}. The conditions of Theorem~\ref{Theo1} hold if $a_{16} \leq 8 = 2^3$, $a_{1i} \leq 8$ and $a_{i6} \leq 8$ for all $2 \leq i \leq 5$. However, an ordinal violation occurs if $a_{16} = 2$, $a_{1i} = 2$ and $a_{i6} = 9$ for $2 \leq i \leq 5$, as well as $a_{1j} = 2$ and $a_{j6} = 2$ for all $2 \leq j \leq 5$, $j \neq i$ according to Example~\ref{Examp1}. The number of these matrices is four as $2 \leq i \leq 5$.

Furthermore, from the set of six comparisons $a_{1i}$ and $a_{i6}$ ($2 \leq j \leq 5$, $j \neq i$), at most one can equal 3 rather than 2 such that the weight of alternative $j$ remains higher than the weight of the best (first) alternative. This gives six additional matrices for each $2 \leq i \leq 5$. Consequently, the number of matrices where the best alternative does not have the greatest weight is $4 + 4 \times 6 = 28$.

Finally, there are 28 analogous matrices where the worst alternative does not have the smallest weight.
\end{proof}

\begin{corollary}
If $K$ matrices are generated with the method of \citet{TuWuPedrycz2023}, the probability that the logarithmic least squares method does not show an ordinal violation is $q = \left( 1 - 56/ 8^9 \right)^K$. \citet{TuWuPedrycz2023} use $K = 10{,}000$, which implies $q \approx 0.9958$. $q < 0.5$ requires at least $1{,}661{,}297$ simulation runs in this framework.
\end{corollary}

\subsection{Limitations} \label{Sec43}

Even though our paper has provided the first sufficient conditions to avoid ordinal violations in a best-worst method matrix, the results have substantial limitations.
First, the theorems are only sufficient but not necessary conditions; one can easily find an example where the uniform lower bound $p$  is close to one, which makes our findings essentially useless.
Second, the restriction on the maximal numerical preference is also universal, but allowing the pairwise comparison between the best and the worst alternatives to be higher than other comparisons may lead to further theorems of a similar flavour.
Third, the sufficient conditions are given only for one particular prioritisation technique, the logarithmic least squares method.

Finally, the issue of ordinal violations has been approached from a purely theoretical perspective. A large database of best-worst method matrices provided by different decision-makers is worth studying in order to estimate the probability of an ordinal violation under several weighting methods.

\subsection{Open questions} \label{Sec44}

Several promising directions exist for future research.
First, if the conditions of Theorem~\ref{Theo1} or Theorem~\ref{Theo2} are not satisfied, then even the logarithmic least squares method may exhibit ordinal violations. It remains to be seen how the probability of an ordinal violation is related to the level of inconsistency in this case.
Second, it would also be interesting to see similar theoretical results for other priority models proposed by \citet{XuWang2024} in the case of best-worst method matrices.
Third, some experiments are needed to check how the decision-makers who evaluate a problem by the best-worst method react to the additional restrictions required by our theorems.

\section*{Acknowledgements}
\addcontentsline{toc}{section}{Acknowledgements}
\noindent
We are grateful to \emph{Lavoslav {\v C}aklovi\'c} for useful remarks. \\
Three anonymous reviewers and seven colleagues provided valuable comments and suggestions on earlier drafts. \\
The research was supported by the National Research, Development and Innovation Office under Grant FK 145838 and the J\'anos Bolyai Research Scholarship of the Hungarian Academy of Sciences.

\bibliographystyle{apalike}
\bibliography{All_references}

\end{document}